\documentclass[11pt,a4paper,reqno, oneside]{amsart}
\usepackage{amsmath}
\usepackage{amsthm}
\usepackage{amsfonts}
\usepackage{amssymb}
\usepackage{amscd}
\usepackage{tikz}
\usepackage{tikz-cd}
\usepackage{comment}
\usepackage{hyperref}
\usepackage{cleveref}
\usepackage{ amssymb }
\usepackage[top=2.5cm]{geometry}
\usepackage[all]{xy}
\usepackage{caption}
\usepackage[english]{babel}
\usepackage{physics}

\theoremstyle{plain}
\newtheorem{thm}{Theorem}[section]

\newtheorem{thmx}{Theorem}

\newtheorem{lem}[thm]{Lemma}
\newtheorem{cor}[thm]{Corollary}

\newtheorem{prop}[thm]{Proposition}
\newtheorem{propx}{Proposition}

\theoremstyle{definition}
\newtheorem{defi}[thm]{Definition}

\newtheorem{rem}[thm]{Remark}
\newtheorem{ex}[thm]{Example}

\newcommand{\R}{\mathbb{R}}
\newcommand{\C}{\mathbb{C}}

\newcommand{\Z}{\mathbb{Z}}

\newcommand{\HF}{\mathrm{HF}}
\newcommand{\CF}{\mathrm{CF}}
\newcommand{\Homol}{\mathrm{H}}

\newcommand{\Symp}{\mathrm{Symp}}

\title{Floer cohomology of Dehn twists \\
along real Lagrangian spheres}
\author{Patricia Dietzsch}
\address{Department of Mathematics\\
  ETH Zürich\\
  Rämistrasse 101, 8092 Zürich, Switzerland}
\email{patricia.dietzsch@math.ethz.ch}

\begin{document}
\maketitle
\begin{abstract}
  We study the Floer cohomology of the Dehn twist
  along a real Lagrangian sphere in a symplectic manifold
  endowed with an anti-symplectic involution.
  We prove that there exists a distinguished element in the Floer group
  that is a fixed point of the automorphism induced by the involution. Our methods of proof are based on Mak-Wu's cobordism and Floer-theoretic considerations.
\end{abstract}

\section{Introduction and Main Results}\label{sec:Intro}
Let $(M,\omega)$ be a closed symplectic manifold and $S\subset M$ a Lagrangian sphere with a parametrization $\iota\colon S^n \xrightarrow{\approx} S$.
Associated to $(S,\iota)$ there exists a distinguished symplectic isotopy class represented by the \textit{Dehn twist}. The Dehn twist $\tau_S$ is a symplectomorphism compactly supported in a neighbourhood of $S$.
Seidel proved that the square of the Dehn twist, in some cases, is not symplectically, but only smoothly isotopic to the identity \cite{seidelthesis97}, \cite{seidel_lectures}.
To prove this result Seidel established a Floer homology exact sequence
\begin{align}\label{I:eqLES}
    \dots \to (\HF^*(S,N) \otimes \HF^*(Q,S))^{k} \to \HF^{k}(Q,N) \to 
    \HF^k(Q,\tau_s(N)) \to \dots
\end{align}
for admissible Lagrangian submanifolds $Q$ and $N$ in $M$
  \cite{seidel03},\cite{seidelbook}.
There is a distinguished element $A\in \HF^*(\tau^{-1}_S)$ that characterizes the map $\HF^k(Q,N) \to HF^k(Q,\tau_S(N))$ that occurs in the sequence. 

Due to the relevance of the above exact sequence it is thus natural to investigate properties of the element $A$.
The goal of this paper is to study the element $A$ in the situation,
where there exists an anti-symplectic involution that preserves $S$.

We work in the following setting. $(M, \omega)$ is a closed
symplectically aspherical symplectic manifold. Unless otherwise explicitely stated, all involved Lagrangian submanifolds are assumed to be closed, oriented and relatively symplectically aspherical. 
Floer cohomology groups are $\mathbb{Z}_2$-graded with coefficients in 
the universal Novikov field over $\mathbb{Z}_2$. More details about these assumptions are given in section \ref{subsec:HFsetting}.

Let $c\colon M \to M$ be an anti-symplectic involution satisfying $c(S)=S$. Consider the smooth involution $\iota^* c:=\iota^{-1}c\iota \colon S^n \to S^n$. We assume that $\iota^*c$ is either smoothly isotopic to the identity or to the reflection
$r(x_1,x_2,\dots,x_{n+1})= (-x_1,x_2,\dots,x_{n+1}).$
This assumption is satisfied in the important geometric setting where $(M,c)$ is a real fiber of a real Lefschetz fibration with one critical point and $S$ is the corresponding vanishing sphere.

Under this assumption, our main result is

\begin{thmx}\label{I:thm:main}
    $c$ induces an automorphism 
    $c_* \colon \HF^*(\tau_S^{-1}) \to \HF^*(\tau_S^{-1})$
    and $c_*(A) = A$.
\end{thmx}

\begin{rem}
$c_* \colon \HF^*(\tau_S^{-1}) \to \HF^*(\tau_S^{-1})$ is an involution
of a vector space over a field with characteristic $2$. Any such map
has a fixed point because $(c_*-\mathrm{id})^2=0$, hence $\mathrm{ker}(c_*-\mathrm{id})\neq 0$. The relevance of the second part of Theorem \ref{I:thm:main} is therefore not merely the existence of a fixed point. It should rather be understood as a special property of the element $A$.
\end{rem}
\subsection{Examples.}\label{subsec:Iexamples}
The assumption on the isotopy class of $\iota^*c$ is automatically satisfied for $n=1,2,3$.
As already mentioned, the assumption is equivalent to $M$ being a real fiber of a real Lefschetz fibration. This is the content of the following
\begin{propx}\label{prop:Lefschetz}
Let $M$ be a symplectic manifold, $S\subset M$ a Lagrangian sphere with parametrization $\iota$ and $c\colon M \to M$ an anti-symplectic involution.
Then the following statements are equivalent:
\begin{enumerate}
    \item[(i)]$c(S)=S$ and $\iota^*c\simeq \mathrm{id}$ or $\iota^*c\simeq r$. \item[(ii)] There exists a real Lefschetz fibration (see Sections \ref{subsec:REALmonodromy} and \ref{subsec:REALdef}) $\pi\colon E \to \mathbb{D}^2$ with real structure $c_E\colon E \to E$, real fiber $M=\pi^{-1}(1)$ and vanishing sphere $(S,\iota)$ such that
$c_E$ restricts to a real structure on $M$ that is Hamiltonian isotopic to $c$.
\end{enumerate}
\end{propx}

Seidel computed Floer cohomology of products of disjoint Dehn twists on surfaces of genus $\geq 2$ in \cite{seidel96}. 
As a special case, his result yields a $\Z$-graded isomorphism
\begin{align}\label{Iexamples:thm:pedrotti}
\HF^*(\tau_S^{-1}) \cong \mathrm{H}^*(M \backslash S; \Lambda).
\end{align}

Later, Gautschi \cite{gautschi} generalised Seidel's result to diffeomorphisms of finite type, still on surfaces.
Recently Pedrotti \cite{pedrotti} proved a $\Z_2$-graded version of (\ref{Iexamples:thm:pedrotti}) for rational, $W^+$-monotone symplectic manifolds of dimension at least $4$.
The $W^+$-condition is explained in Seidel \cite{seidel97}. It is immediate that symplectically aspherical manifolds are $W^+$-monotone.

It turns out that the automorphism $c_*$ on $HF(\tau_S^{-1})$ corresponds to the (topologically induced) map $c^*$ on singular cohomology $\mathrm{H}^*(M\backslash S; \Lambda)$. 
Namely, under the assumption that $M$ is $W^+$-monotone and that 
$c(S)=S$ the following diagram commutes:
\begin{align}\label{HFDehnTwist:diag:action}
\begin{split}
\xymatrix{
&\HF^*(\tau^{-1}_S) \ar[r]^{\cong} \ar[d]^{c_*} &\HF^*(M,S) \ar[d]^{c^*}\\
&\HF^*(\tau_S^{-1}) \ar[r]^{\cong} &\HF^*(M,S).
}
\end{split}
\end{align}
Together with Theorem \ref{I:thm:main} this allows us to deduce topological restrictions on the element $A\in \HF^*(\tau^{-1}_S)$ and sometimes enables us to compute $A$. More concrete examples are explained in section \ref{subsec:REAL2Dexamples}.

\subsection{Outline of Proof of Theorem \ref{I:thm:main}}\label{subsec:Ioutline_of_proof}
We outline the proof of Theorem \ref{I:thm:main}.

We view the Dehn twist as a monodromy in the Lefschetz fibration
$\pi\colon E \to \mathbb{D}^2$ from Proposition \ref{prop:Lefschetz}.
Carrying a result by Salepci \cite{Salepci} over to the symplectic setting, one gets
\begin{propx}\label{R:prop}
$\tau$ is Hamiltonian isotopic to $c \circ \tilde{c}$ for some anti-symplectic involution $\tilde{c}
\colon M \to M$. In particular, $c\tau_S c$ is Hamiltonian isotopic to $\tau_S^{-1}$.
\end{propx}

Floer-theoretic considerations yield a homomorphism
    \[
        c_* \colon \HF^*(\tau_S^{-1}) \to \HF^*(\tilde{c}\tau_S \tilde{c}).
    \]
Proposition \ref{R:prop} implies that $\tilde{c}\tau_S \tilde{c} \simeq \tau_S^{-1}$
and therefore $\HF^*(\tilde{c}\tau_S \tilde{c}) \cong \HF^*(\tau_S^{-1})$.
It follows that $c$ induces an automorphism of $\HF^*(\tau_S^{-1})$,
which proves the first part of Theorem \ref{I:thm:main}.

To show that $c_*(A) = A$, we adopt the framework of Biran-Cornea \cite{BC1}, \cite{BC2}, \cite{BC3}
and Mak-Wu \cite{MakWu} about Lagrangian cobordisms.

Let $M^-$ be the symplectic manifold $(M,-\omega)$.
We denote by $\Gamma_{\phi} \subset M \times M^-$ the graph of $\phi$ for a symplectomorphism $\phi$ on $M$. This is a Lagrangian submanifold of $M\times M^-$.
For $\phi=\mathrm{id}$ it is the diagonal and we write $\Delta:= \Gamma_{\mathrm{id}}$.
In \cite{MakWu} the authors construct a Lagrangian cobordism $V_{MW} \subset M\times M^- \times \C$
that has three ends: $S\times S, \Delta$ and $\Gamma_{\tau_S^{-1}}$.
We recall the construction of $V_{MW}$ in section \ref{sec:MW}.
By general results on Lagrangian cobordisms due to Biran-Cornea this cobordism induces an exact triangle in $D\mathcal{F}uk(M\times M^-)$:

\begin{center}
\begin{tikzpicture}[node distance=1cm, auto]
\node (P) {$S\times S$};
\node (phantom)[right of =P] {};
\node(Q)[right of=phantom, below of=phantom] {$\Delta$};
\node (qhantom)[left of =Q] {};
\node (B) [left of =qhantom, below of=qhantom] {$\Gamma_{\tau_S^{-1}}$};

\draw[->](P) to node {}(Q);
\draw[->](Q) to node {}(B);
\draw[->] (B) to node {} (P);
\end{tikzpicture}
\end{center}
The associated long exact sequence is 
\begin{align}\label{I:eqLES2}
    \dots \to \HF^k(K,S\times S) \to \HF^k(K,\Delta) \to \HF^k(K,\Gamma_{\tau_S^{-1}}) \to \HF^{k+1}(K,S\times S) \to \dots,
\end{align}
where $K$ is an admissible Lagrangian submanifold in $M\times M^-$.
For the special case $K=Q\times N$, this sequence reduces to Seidel's long exact sequence (\ref{I:eqLES}).
The middle map in sequence (\ref{I:eqLES2}) can be understood as $\mu^2(A,-)$ for the element $$A\in \HF^0(\Delta, \Gamma_{\tau_S^{-1}}) \cong \HF^0(\tau_S^{-1}).$$

Consider the symplectomorphism
\begin{align*}
\Phi \colon M \times M^- \times \C &\longrightarrow M \times M^- \times \C \\
(x,y,z) &\longmapsto (c(y),c(x),z).
\end{align*}
$\Phi$ preserves the ends of the cobordisms $V_{MW}$. In particular,
$\Phi$ induces an automorphism 
$$\Phi_* \colon \HF^*(\Delta, \Gamma_{\tau^{-1}_S}) \to \HF^*(\Delta, \Gamma_{\tau^{-1}_S}).$$
This automorphism corresponds to the action of $c$ on $\HF(\tau^{-1}_S)$, namely
the following diagram commutes
\begin{align}\label{Ioutline:diag:phi_c}
\xymatrix{
&\HF^*(\tau^{-1}_S) \ar[r]^{c_*} \ar[d]^{\cong} &\HF^*(\tau^{-1}_S) \ar[d]^{\cong}\\
&\HF^*(\Delta, \Gamma_{\tau^{-1}_S}) \ar[r]^{\Phi_*} &\HF^*(\Delta, \Gamma_{\tau^{-1}_S}).
}
\end{align}
We explain these isomorphisms and the commutativity of the diagram in section \ref{sec:HF}.
A major step in the proof is the following 
\begin{thmx}\label{Ioutline_of_proof:thm:invariance}
$\Phi(V_{MW})$ is Hamiltonian isotopic to $V_{MW}$.
\end{thmx}

We show how this implies Theorem \ref{I:thm:main}.
Denote by $\bar{A}\in \HF^*(\Delta, \Gamma_{\tau^{-1}_S})$ the element corresponding
to $A\in \HF^*(\tau^{-1}_S)$ under the natural isomorphism
$\HF(\tau^{-1}_S) \cong \HF(\Delta, \Gamma_{\tau^{-1}_S})$.
As a consequence of Theorem \ref{Ioutline_of_proof:thm:invariance}, the cobordisms $V_{MW}$ and $\Phi(V_{MW})$
induce isomorphic triangles. In particular, the following diagram commutes:
\begin{align*}
\xymatrix@C+2pc{
&\HF^*(K,\Delta) \ar[r]^{\mu^2(\bar{A}, -)} \ar[d]^{\Phi_*} & \HF^*(K, \Gamma_{\tau^{-1}_S}) \ar[d]^{\Phi_*}\\
&\HF^*(K, \Delta) \ar[r]^{\mu^2(\Phi_*(\bar{A}), -)} & \HF^*(K, \Gamma_{\tau^{-1}_S})
}
\end{align*}
for all $K$. 
It follows that $\Phi_*(\bar{A})= \bar{A}$ and hence $c_*(A) =A$ by commutativity of diagram (\ref{Ioutline:diag:phi_c}).

\begin{rem}\label{Ioutline_of_proof:rmk:setting}
\begin{enumerate}
    \item 
It can be seen from the proof that all is needed are well-defined Floer cohomology groups, and applicability of Biran-Cornea's \cite{BC1},\cite{BC2} and Mak-Wu's \cite{MakWu} framework.
One could therefore easily weaken the asphericity assumption to monotonicity conditions. 
    \item The assumption that $M$ is closed is important for our arguments:
    The version of Floer cohomology we use only works for compactly supported symplectomorphisms. In general however, the monodromy in a Lefschetz fibration with non-compact fibers, if it exists, is not compactly supported. 
    We expect that the results generalize to a non-compact framework,
    when working with an appropriate version of Floer theory.
    \item In general we have a symplectic isotopy
\[
 c\circ \tau_{(S,\iota)} \circ c \simeq \tau_{(S,c\circ \iota)}^{-1}.
\]
However, it is unknown how the Dehn twist depends on the parametrization of the sphere. It is only known that if $\iota^*c$ is isotopic to an isometry, then the Dehn twist associated to $c\circ \iota$ is symplectically isotopic to the Dehn twist associated to $\iota$
\cite[Remark 3.1]{seidelthesis97}.
This explains why we make the assumption on the mapping class of 
$\iota^*c$.
    \item     
The second map in the long exact sequence (\ref{I:eqLES}) is 
\[
\mu^2(a_N,-) \colon \HF^k(Q,N) \to \HF^k(Q,\tau_S(N))
\]
for some element $a_N\in \HF^0(N,\tau_S(N))$.
$a_N$ and $A\in \HF^*(\tau_S^{-1})$ are related as follows.
There is an operation
\[
    * \colon \HF^*(\tau_S^{-1}) \otimes \HF^*(N,N) \to \HF^*(N, \tau_S(N)).
\]
If $e_N \in HF^*(N,N)$ denotes the unit, we have $A * e_N = a_N$.
The fixed point property $c_*(A) = A$ then implies 
\begin{align}\label{eq:seidel}
     \gamma(a_N) = a_{c(N)},
\end{align}
where $\gamma$ is the isomorphism
\begin{align*}
\HF^*(N,\tau_S(N)) 
\cong 
\HF^*(\tilde{c}(N), c(N)) \cong
\HF^*(c(N), \tau_S(c(N)).
\end{align*}
The construction of $a_N$ is explained in \cite[Sections 17a-17c]{seidelbook}.
$a_N$ comes from counting the number of holomorphic sections of a Lefschetz fibration with moving boundary condition coming from moving $N$ via parallel transport. The invariance property (\ref{eq:seidel}) can be proven directly in Seidel's
framework, by observing that the holomorphic sections for boundary conditions coming from $N$ and $c(N)$ are in bijection.
\end{enumerate}
\end{rem}
\subsection{Organisation of the Paper.}\label{subsec:Iorg}
The rest of this paper is organised as follows.
In section \ref{sec:REALlefschetz} 
we explain the construction of real Lefschetz fibrations and the decomposition of the monodromy into two anti-symplectic involuions as stated in Propositions \ref{prop:Lefschetz} and \ref{R:prop}.
In section \ref{sec:HF} we fix the setting and collect the properties of Floer cohomology we need.
In section \ref{sec:LC} we briefly recall Biran-Cornea's Lagrangian cobordism framework and how cobordisms induce cone decompositions.
Section \ref{sec:MW} recalls the construction of the Mak-Wu cobordism. In section 
\ref{sec:INV} we prove Theorem \ref{Ioutline_of_proof:thm:invariance} about the symmetry of the cobordism. 
Section \ref{sec:ADD} contains some more background material on Floer cohomology
for the convenience of the reader.
The \hyperref[appendix]{appendix} contains some algebraic background on Fukaya categories.

\section{Dehn twist and real Lefschetz fibrations.}\label{sec:REALlefschetz}
In this section we show Propositions \ref{prop:Lefschetz} and \ref{R:prop}.
This is based on work by Salepci \cite{Salepci} on real Lefschetz fibrations in the smooth setting.
Since we keep the discussion here relatively brief, we refer the interested reader to the following references for a more detailed treatment of (real) Lefschetz fibrations:
\cite{seidelbook, BC3, salepci2, keating}.
\subsection{Dehn twist.}\label{subsec:REALDehntwist}

Let $S \subset M$ be a Lagrangian sphere together with an embedding
$\varphi \colon S^n \to M$ of the $n$-dimensional sphere $S^n$ with image $S$.
We refer to $(S,\varphi)$ as a parametrized Lagrangian sphere. 
\footnote{Seidel uses the word ``framed sphere" for this situation in \cite{seidelbook}.}
The Dehn twist $\tau_S$ along $S$ is a symplectomorphism compactly supported in a neighbourhood of $S$. It is defined up to symplectic isotopy. The precise map will depend on a Dehn twist profile function and on a Weinstein neighbourhood of $S$.
As explained in \cite[Proposition 2.3]{seidelthesis97} the symplectic isotopy class of $\tau_S$ is independent of $\varphi$ in dimension $4$.
In general however, it might depend on the parametrization \cite[Remark 3.1]{seidelbook}.
We briefly recall the definition, following closely the exposition in \cite{MakWu}.
\begin{defi}
Let $\epsilon > 0$.
A \textit{Dehn twist profile function} is a smooth function
$$\nu_\epsilon^{Dehn} \colon \R_{\geq0} \longrightarrow \R$$
satisfying
\begin{align*}
\begin{cases}
	 \nu_\epsilon^{Dehn}(r) = \pi - r \qquad \qquad \qquad \qquad \qquad
                                       &\text{for } 0\leq r << \epsilon,\\
	 0 < \nu_\epsilon^{Dehn}(r) < \pi \text{ and strictly decreasing}
	 									\qquad &\text{for } 0< r < \epsilon,\\
	 \nu_\epsilon^{Dehn}(r)=0 \qquad \qquad \qquad \qquad \qquad \qquad
	                                 &\text{for } r\geq \epsilon.
\end{cases}
\end{align*}
\end{defi}

Consider the canonical Riemannian metric on $S^n$.
We have a canonical isomorphism $T_*S^n\cong T^*S^n$ and we denote by $\norm{\xi}$ the norm of the tangent vector identified with $\xi \in T^*S$. We denote by 
\[
T_r^*S^n =  \left \{ \xi \in T^*S^n \; \mid \; \norm{\xi} < r \right \}
\]
the open subset of $T^*S^n$ consisting of cotangent vectors of norm strictly less than $r$.

Let $V\subset M$ be a Weinstein neighbourhood of $S$ together with a symplectic embedding
$$\varphi \colon V \longrightarrow T^*S^n$$
that identifies $S\subset V$ with the zero-section $0_{S^n} \cong S^n$ via $\iota$ and $\varphi(V) = T_\epsilon^*S^n$ 
for some $\epsilon>0$. 

\noindent
Consider the continuous function $\sigma \colon T^*S^n \longrightarrow \R, \, \sigma(\xi) = \norm{\xi}$. This function is not smooth on the zero-section $0_{S^n}$, but has a well-defined Hamiltonian flow on the complement:
$$\psi^{\sigma}_t \colon \left(T^*S^n\right) \backslash 0_{S^n} \longrightarrow \left(T^*S^n \right) \backslash 0_{S^n}.$$
\begin{defi}
The \textit{model Dehn twist} on $T^*S^n$ is the
diffeomorphism defined by
\begin{align*}
\tau_{S^n} \colon T^*S^n &\longrightarrow T^*S^n, \\
\xi &\longmapsto 
\begin{cases}
	\psi_{\nu_\epsilon^{Dehn}(\sigma(\xi))}^\sigma (\xi) &\text{ for } \xi \notin 0_{S^n},\\
	-x \qquad \qquad &\text{ for } \xi=x\in S^n.
\end{cases}
\end{align*}
The \textit{Dehn twist} in $M$ along $S$ is then given by copying the model Dehn twist into $V$ via $\varphi$:
\begin{align*}
\tau_S= 
\begin{cases}
    \varphi^{-1} \circ \tau_{S^n} \circ \varphi \qquad &\text{on $V$}\\
    \mathrm{id} \qquad  &\text{on $M\backslash V$}.
\end{cases}
\end{align*}
\end{defi}

\subsection{The Dehn twist as a monodromy.}\label{subsec:REALmonodromy}
We adopt here the definition used in \cite{BC3}. We denote by $\mathbb{D}^2$ the closed unit disc viewed as a subset of $\C$. A Lefschetz fibration with base $\mathbb{D}^2$ consists of 
\begin{enumerate}
    \item a closed symplectic manifold $(E,\Omega_E)$ endowed with an almost complex structure $J_E$,
    \item a proper $(J_E,i)$-holomorphic map $\pi \colon E \to \mathbb{D}^2$
\end{enumerate}
such that 
\begin{enumerate}
    \item $\pi$ has only finitely many critical points with distinct critical values,
    \item all the critical points of $\pi$ are ordinary double points, that is for every critical point $p\in E$, there exists $J_E$-holomorphic coordinates around $p$ such that in these coordinates $\pi(z_1, \dots, z_n) = z_1^2 + \dots + z_n^2$ holds.
\end{enumerate}
For $p\in \mathbb{D}^2$ we denote by $E_p:= \pi^{-1}(\{p\})$ the fiber above $p$.
All regular fibers of $\pi$ are symplectic manifolds with symplectic form induced from $\Omega_E$.

Given a symplectic manifold $(M,\omega)$ and a parametrized Lagrangian sphere $S$, one can construct a Lefschetz fibration with smooth fiber $M$ and vanishing sphere $S$ such that the Dehn twist is symplectically isotopic to the monodromy around a critical point.
We refer the reader to \cite[Section 1]{seidel03} and \cite[Section (16e)]{seidelbook} for a detailed explanation.
We only include a very brief outline of the construction here.
Consider the following local model for $\epsilon > 0$: Let $Q \colon \C^{n+1} \to \C, Q(z_1, \dots, z_{n+1}) = z_1^2 + \dots + z_{n+1}^2$ and define the total space of the fibration to be
\[
    E^0_{\epsilon} := \left \{ z\in \C^{n+1} \, \big \vert \, \lvert Q(z) \rvert \leq 1,
    \frac{\vert z \vert ^4 - \vert Q(z) \vert ^2}{4} < \epsilon \right \}.
\]
The fibration then is $\pi_{\epsilon}^0 \colon E^0_{\epsilon} \to \mathbb{D}^2, \pi(z) = Q(z).$
The symplectic form on $E^0_{\epsilon}$ is of the form $\Omega_0 + \mathrm{d}\gamma$, where $\Omega_0$ is the standard symplectic form on $\C^{n+1}$
and $\gamma$ a certain $1$-form whose precise form is not relevant to us. Its effect is, that the fibration $\pi^0_{\epsilon}$ is trivial near the boundary.
The smooth fibers are symplectomorphic to $T^*_{\epsilon}S^n$.
Consider the family of Lagrangian spheres
\[
    \Sigma_r = \sqrt{r}S^{n} = \{ (\sqrt{r}z_1, \dots , \sqrt{r}z_{n+1}) \, \big \vert \, z \in S^n \subset \R^{n+1}\}  \subset \left(E_{\epsilon}^0\right)_{r}
\]
for $r>0$. They are called vanishing cycles.
The union $\Sigma = \left(\cup_{r>0} \Sigma_r\right) \cup \{0\}$ is a Lagrangian disc in $E^0_{\epsilon}$, called a Lefschetz thimble.
There is an isomorphism 
\[
    \Phi \colon E^0_{\epsilon}\backslash \Sigma \to \mathbb{D}^2 \times
    (T_{\epsilon}^*S^n \backslash S^n).
\]
The monodromy $\tau \colon (\pi_{\epsilon}^0)^{-1}({1}) \to (\pi_{\epsilon}^0)^{-1}({1})$ along $\partial \mathbb{D}^2$ is the Dehn twist along the vanishing cycle $\Sigma_1$ \cite[Lemma 1.10]{seidel03}.
To get the claimed Lefschetz fibration $\pi^0 \colon E^0 \to \mathbb{D}^2$, one glues $E^0_{\epsilon}$ together with the trivial fibration $\mathbb{D}^2 \times (M \backslash V)$ via $\varphi$.

Locally, each Lefschetz fibration looks like a model Lefschetz fibration $E^0$.
In particular, there is a notion of vanishing spheres in any Lefschetz fibration.
The monodromy $\tau \colon E_p \to E_p$ along a path around the singularity is the Dehn twist along a vanishing cycle in $E_p$.
Usually, the monodromy in not supported near $S$. However, $\tau$ is symplectically isotopic to the Dehn twist as defined
in section \ref{subsec:REALDehntwist}.

\subsection{Real Lefschetz fibrations.}\label{subsec:REALdef}
A Lefschetz fibration $\pi\colon E \to \mathbb{D}^2$ is called real, if the total space $E$ is endowed with an anti-symplectic involution $c_E\colon E \to E$ that covers complex conjugation $c_{\C}\colon \mathbb{D}^2 \to \mathbb{D}^2$, meaning the diagram
\begin{align}\label{REALdef:diag:conj}
\begin{split}
\xymatrix{
&E \ar[r]^{c_E} \ar[d]^{\pi} &E \ar[d]^{\pi}\\
&\mathbb{D}^2 \ar[r]^{c_{\C}} &\mathbb{D}^2
}
\end{split}
\end{align}
commutes.
Consider the fiber $M=E_1$ over $1$. $c_E$ induces an anti-symplectic involution on $c$ on $M$. The following Lemma shows that the assumption of Theorem
\ref{I:thm:main} is satisfied.
\begin{lem}\label{lem:lefschetz}
    $c(S)=S$ and $\iota^{-1}c\iota \simeq \mathrm{id}$ or $\iota^{-1}c\iota \simeq r$,
    where $\iota$ is the canonical framing of the vanishing sphere $S$.
\end{lem}
\begin{proof}
    $c(S)=S$ follows from $c_E(0)=0$ and the fact that $c_E$ commutes with parallel transport.
    For the second part, note that it is enough to consider the model
    $Q \colon \C^{n+1} \rightarrow \C$.
    In that case, $S=S^{n}\subset \C^{n+1}$ is a standard sphere.
    Note that $c_E$ restricted to the thimble
    $\Sigma= B^{n+1}(0)$ is a smooth extension of the sphere $c\vert_{S^n}$ to the ball.
    Moreover, since parallel transport commutes with $c_E$, it is a linear extension, in the sense that
    \[
        c_E(x)=c\left(\frac{x}{\norm{x}}\right)\norm{x} .
    \]
    It follows that $c_E\vert_{\R^{n+1}}$ is an orthogonal linear transformation and hence
    $c\vert_{S^n}$ is an isometry. In particular, $c\vert_{S^n}$ is smoothly isotopic to $\mathrm{id}$
    or $r$.    
\end{proof}

Proposition \ref{prop:Lefschetz} states that the condition on $\iota^*c$ is equivalent to $M$ being the fiber of a real Lefschetz fibration. One direction is proven in Lemma \ref{lem:lefschetz} above. We now prove the other direction.
\begin{proof}[Proof of Proposition \ref{prop:Lefschetz}.]
Suppose the tuple $(M,S,\iota,c)$ satisfy the conclusion of Lemma \ref{lem:lefschetz}. We want to construct a real Lefschetz fibration whose fiber is $M$, whose vanishing sphere is $(S,\iota)$ and whose real structure restricts to a real structure Hamiltonian isotopic to $c$. 

First we endow the Lefschetz fibration $\pi^0_{\epsilon} \colon E^0_{\epsilon} \to \mathbb{D}^2$ from the previous section with a real structure.
We consider two options:
\[
    c_1(z_1,\dots,z_{n+1}) = (\overline{z_1}, \dots,\overline{z_{n+1}})
\]
and 
\[
     c_2(z_1,z_2,\dots,z_{n+1}) = (-\overline{z_1},\overline{z_2},\dots, \overline{z_{n+1}}).
\]
These are real structures on $E^0_{\epsilon}$.

By Proposition \ref{INVlinear:prop:main} there exists a Hamiltonian isotopy $\psi_t$ on $M$
supported in $V$
such that in the model $T^*_{\delta}S^n$ ($\delta<\epsilon$ small enough) one has 
\[
    \psi_1 c (q,p) = (q,-p)
\]
if $\varphi^*c\simeq \mathrm{id}$
and
\[
    \psi_1 c (q,p) = (r(q),-r(p))
\]
if $\varphi^*c \simeq r$.
These two maps exactly correspond to $c_1$ and $c_2$ on the fiber $(\pi^0_{\epsilon})^{-1}(1)$. 

We now glue the fibration $\pi \colon E^0 \to \mathbb{D}^2$ from two parts: the trivial fibration 
$$\mathbb{D}^2 \times (M \backslash \varphi^{-1}(T_{\delta}^*S^n))$$ and the local model fibration $E_{\delta}^0$.
On the first part, we define
$c(z, x) := (\overline{z},\psi_1c(x))$. On $E_{\delta}^0$ we define $c(z):= c_1(z)$ or $c(z)=c_2(z)$.
These definitions are compatible on the glued region and hence descend to a real structure $c_{E^0}$ on $E^0$ satisfying $c_E\vert_{E_1}=\psi_1c$.
\end{proof}
\subsection{Splitting of the monodromy into anti-symplectic involutions.}\label{subsec:REALsplitting}
Let $$\pi \colon E \to \mathbb{D}^2$$ be a real Lefschetz fibration with real structure $c_E \colon E \to E$ as above. 
We assume that $p\in E$ is the unique critical point of $\pi$ and $\pi(p)=0$.
Let $M:=E_1:=\pi^{-1}(\{1\})$ and denote
by $\tau\colon M \to M$ the monodromy along the boundary loop 
$\gamma(t) = e^{2\pi i t}, t\in [0,1].$
The following result is due to Salepci \cite{Salepci} in the smooth category.
Here we adapt it to the symplectic framework.
\begin{lem}\label{EXlem:splitting}
$\tau$ splits into a product of two anti-symplectic involutions on $M$.
More concretely, $\tau = c_+ \circ c_-$
for two anti-symplectic involutions $c_{\pm}\colon M \longrightarrow M$, where $c_+= (c_E)\vert _{E_1}$. 
Equivalently, we have $c\tau c = \tau^{-1}$.
\end{lem}
\begin{proof}
$\Omega_E$ defines a symplectic connection on the smooth part of $E$.
Let us denote by 
\[
    P_{\gamma(s);t} \colon E_{\gamma(s)} \to E_{\gamma(s+t)}
\]
the parallel transport for time $t$ along $\gamma$. 
Let $v\in E_{-1}$. Consider the parallel lift $w(t)\in E_{e^{\pi i - \pi i t}}$
of $x:=c_E(v)$ along the upper half $\eta^+$ of $\gamma$. 
Note that 
$$c_E\circ (P_{1;\frac{1}{2}})^{-1} \circ c_E(v) =c_E(w(1)).$$
It is straight-forward to check that $v(t):= c_E(w(t))$ is actually a 
parallel lift of $v$ along the lower half $\eta^-$ of $\gamma$. This uses $(dc_E)(H_{w(t)}) = H_{c_E(w(t))}$.
Hence,
$$c_E \circ (P_{1;\frac{1}{2}})^{-1} \circ c_E = P_{-1,\frac{1}{2}}$$
and the lemma follows:
$$\tau = P_{-1;\frac{1}{2}} \circ P_{1;\frac{1}{2}} = 
c_E \circ (P_{1;\frac{1}{2}})^{-1} \circ c_E \circ P_{1;\frac{1}{2}}
= c_+ \circ c_-,$$
where $c_+ = (c_E)\vert_M$ and $c_- = (P_{1;\frac{1}{2}})^{-1} \circ c_E \circ P_{1;\frac{1}{2}}$.
\end{proof}

This proves Proposition \ref{R:prop}.
Alternatively, Proposition \ref{R:prop} can also be shown directly from the definition of a model Dehn twist without going through real Lefschetz fibrations.
\subsection{Examples in $2$ dimensions.}\label{subsec:REAL2Dexamples}
\begin{ex}[Genus $2$ surface]
Let us consider the genus $2$ surface $\Sigma_2$.
Take $S$ to be a separating curve, going once around between the two holes,
as in Figure \ref{fig:genus2}. Consider the Dehn twist $\tau_S$ around $S$.
\begin{figure}[ht]
\centering
\begin{tikzpicture}
\draw[smooth] (0,1) to[out=30,in=150] (2,1) to[out=-30,in=210] (3,1) to[out=30,in=150] (5,1) to[out=-30,in=30] (5,-1) to[out=210,in=-30] (3,-1) to[out=150,in=30] (2,-1) to[out=210,in=-30] (0,-1) to[out=150,in=-150] (0,1);
\draw[smooth] (0.4,0.1) .. controls (0.8,-0.25) and (1.2,-0.25) .. (1.6,0.1);
\draw[smooth] (0.5,0) .. controls (0.8,0.2) and (1.2,0.2) .. (1.5,0);
\draw[smooth] (3.4,0.1) .. controls (3.8,-0.25) and (4.2,-0.25) .. (4.6,0.1);
\draw[smooth] (3.5,0) .. controls (3.8,0.2) and (4.2,0.2) .. (4.5,0);
\draw[blue] (-0.5,0) arc(180:0:0.51 and 0.2);
\draw[blue, dashed] (-0.5,0) arc(180:0:0.51 and -0.2);
\filldraw[blue] (-0.5,0) circle (0pt) node[anchor=east]{$\alpha_1$};
\draw[red, thick] (2.5,-0.85) arc(270:90:0.3 and 0.85);
\draw[red,dashed,thick] (2.5,-0.85) arc(270:450:0.3 and 0.85);
\filldraw[red] (2.6,0) circle (0pt) node[anchor=east]{$S$};
\draw[blue] (5.5,0) arc(180:0:-0.51 and 0.2);
\draw[blue, dashed] (5.5,0) arc(180:0:-0.51 and -0.2);
\filldraw[blue] (5.5,0) circle (0pt) node[anchor=west]{$\alpha_2$};
\draw[green] (1,0) ellipse (1 and 0.5);
\filldraw[green] (1,1) circle (0pt) node[anchor=north]{$\beta_1$};
\draw[green] (4,0) ellipse (1 and 0.5);
\filldraw[green] (4,1) circle (0pt) node[anchor=north]{$\beta_2$};
\end{tikzpicture}
\captionsetup{justification=centering}
    \caption{Genus $2$ surface with Lagrangian sphere $S$.}
    \label{fig:genus2}
\end{figure}
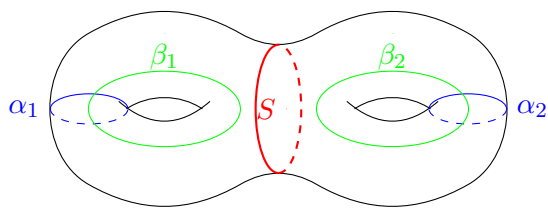

As in \cite{seidel96} we can work over $\Z_2$ instead of the Novikov field,
and the Floer cohomology groups are $\Z$-graded.

$\tau_S$ splits into the product of two anti-symplectic involutions:
Take $c$ to be the anti-symplectic involution which is a reflection along $S$.
It is straight forward to check that $\tilde{c}:=c \circ \tau_S$ is an anti-symplectic involution.
\noindent In particular, we can write $\tau_S = c \circ \tilde{c}$.

Let us compute $c_*\colon HF^*(\tau^{-1}_S) \to HF^*(\tau^{-1}_S)$.
By the isomorphism (\ref{Iexamples:thm:pedrotti}) Floer cohomology of $\tau^{-1}_S$ is
\begin{align*}
 HF^*(\tau^{-1}_S) &\cong H^*(\Sigma \backslash S; \Z_2)\\
    &\cong  H^*(\Sigma \backslash S;\Z_2) \\
    &\cong H^*(S^1 \vee S^1;\Z_2)
     \oplus H^*(S^1 \vee S^1;\Z_2)\\
     &\cong \Z_2 [pt_1] \oplus \Z_2 \alpha_1 \oplus \Z_2 \beta_1
     \oplus \Z_2 [pt_2] \oplus \Z_2 \alpha_2 \oplus \Z_2 \beta_2.
\end{align*}

In degree 
$0$, the matrix representing $c^*$ on $H^0(\Sigma \backslash S; \Z_2)$ with respect to the basis $[pt_1], [pt_2]$
is
\begin{align*}
\begin{pmatrix}
 0 & 1 \\
1 & 0 
\end{pmatrix}.
\end{align*}
It follows from Theorem \ref{I:thm:main} that $A = [pt_1] + [pt_2]$.
\end{ex}

\begin{ex}[higher genus surfaces] 
Similarly, we can consider any surface $\Sigma$ of genus $g\geq 2$, $S$ a separating circle in it that is the fixed point set of a reflection. Then $HF^0(\tau_S^{-1})
\cong \Z_2 \oplus \Z_2$, where each of the two summands corresponds to one of the connected components of $\Sigma \backslash S$. Theorem \ref{I:thm:main} implies $A= (1,1)$.
\end{ex}

\begin{ex}[Torus]
Let $S$ be any non-contractible embedded circle in the torus $T^2$.
Using the long exact sequence (\ref{I:eqLES2}) applied to $K = \Delta$ one computes
\[
    HF^0(\tau_S^{-1}) \cong H^{even}(T^2; \Lambda) / H^0(S; \Lambda) \cong H^2(T^2;\Lambda) \cong \Lambda.
\]
For any anti-symplectic involution $c\colon T^2 \to T^2$ satisfying
$c(S)=S$, it follows that $c_* = \mathrm{id}$.

\end{ex}

\section{Floer cohomology}\label{sec:HF}
In this section we collect the main properties of Floer cohomology we need in the sequel.
\subsection{Setting.}\label{subsec:HFsetting}
We assume that $M$ is symplectically aspherical, that is 
for every smooth map $u \colon S^2 \to M$, its symplectic area vanishes:
\[
    \int_{S^2}u^*\omega = 0.
\]
Moreover, we assume that all involved Lagrangian submanifolds are relatively symplectically aspherical, that is for every smooth map $u\colon D^2 \to M$ satisfying $u(\partial D^2) \subset L$, we have
\[
    \int_{D^2} u^*\omega =0.
\]
In particular, $S\subset M$ is relatively symplectically aspherical. 
This is automatic if $M$ is symplectically aspherical, unless $S$ has dimension $1$. In the latter case, the condition is equivalent to $S$ being a non-contractible circle. 
In this situation, Floer cohomology $\HF^*(f)$ for a symplectomorphism $f \in \Symp(M)$, and Lagrangian Floer cohomology $\HF^*(L,K)$ for Lagrangians $L,K$ as above can be defined over the universal Novikov field 
\[
\Lambda = \left\{ \sum a_k q^{\omega_k} \bigg \vert \vert a_k\in \Z_2, \omega_k \in \R,
\lim\limits_{k\to \infty}\omega_k = \infty
\right \}.
\]
$\HF^*(f)$ and $\HF^*(L,K)$ are $\Z_2$-graded, whenever $L$ and $K$ are oriented.
We include a section about the definition of these groups for convenience of the reader in section
\ref{sec:ADD}. For a more detailed exposition, we refer the reader to \cite{DostSal, seidelthesis97, lee}
for $\HF^*(f)$ and to \cite{floer, oh1, oh_add} for $\HF^*(L,K)$ .

\subsection{Conjugation invariance.}\label{subsec:HFconjugation_invariance}
Let $f$ be a symplectomorphism on $X$ and $\varphi$ be an antisymplectic diffeomorphism on $X$.
We will make substancial use of the following fact, which is an anti-symplectic version of the well-known conjugation invariance of Floer cohomology (see e.g. \cite[section 3]{seidel_lectures}).
We include a proof in section \ref{subsec:AdddefHF}.
\begin{prop}\label{HFconjugation_invariance:thm:main}
There is a canonical graded isomorphism
$$(\varphi f^{-1})_*\colon \HF^*(f^{-1}) \to \HF^*(\varphi f \varphi^{-1}).$$
\end{prop}

If $\tau_S = c \circ \tilde{c}$ we can apply this result to $\varphi = \tilde{c}$ and $f= \tau_S$.
We get an automorphism
\begin{align*}
c_* \colon \HF^*(\tau_S^{-1}) \longrightarrow \HF^*(\tilde{c}\tau_S \tilde{c}) \cong \HF^*(\tau_S^{-1}).
\end{align*}
This is induced by the chain-level map sending a generator $x$ to $\tilde{c}\tau_S^{-1}(x) = c(x)$, concatenated with a continuation map.

\subsection{Lagrangian Floer cohomology.}\label{subsec:HFlag}
Note that for any symplectomorphism $f$ on a symplectically aspherical symplectic manifold $M$, the graph $\Gamma_{f}$ is a relatively aspherical Lagrangian manifold in $M\times M^{-}$. Also, products of relatively aspherical Lagrangians in $M$ are relatively aspherical Lagrangians in $M\times M^-$.

We endow the graph $\Gamma_{f}$ with the following orientation:
Given a positive basis $v_1, \dots , v_{2n}$ of $T_xM$, then the basis
$(v_1,Df_x(v_1)), \dots , (v_{2n}, Df_x(v_{2n}))$ of $T_x\Delta \subset T_xM \oplus T_xM$ is defined to be positive
if $(-1)^{\frac{n(n-1)}{2}}=1$ and negative otherwise, see \cite{WW}.
Moreover, given an oriented Lagrangian $N$, note that $f(N)$ has a canonical orientation.

Let $Q$ and $N$ be oriented Lagrangians in $M$. There are the following canonical graded isomorphisms between
Floer cohomology groups for Lagrangians in $M\times M^{-}$ and Lagrangians in $M$:
\begin{enumerate}
    \item $\HF^*(Q\times N, \Gamma_{f^{-1}}) \cong \HF^*(Q, f(N))$
    \item $\HF^*(Q\times N, Q'\times N') \cong \HF^*(Q,Q') \otimes \HF^*(N',N)$
\end{enumerate}

\subsection{Floer cohomology as a special case of Lagrangian Floer cohomology.}\label{subsec:HFcompare}
Floer cohomology of a symplectomorphism $f$ can be viewed as
Lagrangian Floer cohomology of the pair $(\Delta, \Gamma_{f})$.
This isomorphism is well-known, see for instance \cite{WW}, \cite{MakWu} and \cite[section 2.7]{leclercq-zapolsky}.
Namely we have
\begin{prop}\label{HFcompare:prop}
There is a canonical graded isomorphism $\Psi_{f} \colon \HF(f) \to \HF(\Delta, \Gamma_{f})$.
\end{prop}
\noindent
For the convenience of the reader we include a sketch of the proof in section \ref{sec:ADD}.

Let $\varphi \colon M \to M$ be an anti-symplectic involution.
Consider the symplectomorphism
\begin{align*}
     \Phi^{\varphi} \colon M \times M^- \longrightarrow M \times M^- \\
     (x,y) \longmapsto (\varphi(y), \varphi(x)).
\end{align*}
The map $(\varphi f^{-1})_*$ on $\HF^*(\tau_S^{-1})$ corresponds to $\Phi^{\varphi}_*$
under the isomorphism of Proposition \ref{HFcompare:prop}, i.e. the following diagram commutes:
\begin{align*}
    \xymatrix{
&\HF^*(f^{-1}) \ar[r]^{\left(\varphi f^{-1}\right)_*} \ar[d]^{\Psi_{f^{-1}}} &\HF^*(\varphi f \varphi^{-1}) \ar[d]^{\Psi_{\varphi f \varphi ^{-1}}}\\
&\HF^*(\Delta, \Gamma_{f^{-1}}) \ar[r]^{\Phi^{\varphi}_*} &\HF^*(\Delta, \Gamma_{\varphi f \varphi ^{-1}})
}
\end{align*}
As a special case, we recover the commutative diagram (\ref{Ioutline:diag:phi_c}) by setting $f = \tau_S$ and $\varphi = \tilde{c}$.

\section{Lagrangian cobordisms}\label{sec:LC}
\subsection{Definition of a Lagrangian cobordism.}\label{subsec:LCdefinition}
In this section we recall the definition of Lagrangian cobordisms as studied by Biran and Cornea in the series of papers \cite{BC1, BC2, BC3}.
Let $(M,\omega)$ be a symplectic manifold.
Consider the product symplectic manifold $(M \times \R^2, \omega \oplus \omega_{\mathrm{std}})$.
Here, $\omega_{\mathrm{std}})= \mathrm{d}x \wedge \mathrm{d}y$ denotes the standard symplectic form on $\R^2$. We denote by $\pi \colon M \times \R^2 \to \R^2$ the projection to the plane.
For subsets $V \subset M \times \R^2$ and $Z\subset \R^2$, we write $V\vert_Z := V \cap \pi^{-1}(Z)$ for the restriction of $V$ over $Z$. 
A Lagrangian submanifold $V\subset M\times \R^2$ is called a \textit{Lagrangian cobordism} if there exists $R>0$ such that
\begin{enumerate}
\item[(i)] $$V \vert_{(-\infty, -R] \times \mathbb{R}} = \bigcup_{j=1}^{k_-} L_j \times (-\infty,-R] \times \{ j\}$$
	for some closed Lagrangian submanifolds $L_1, \dots, L_{k_-} \subset M$,
\item[(ii)] $$V \vert_{[R, \infty) \times \mathbb{R}} = \bigcup _{j=1}^{k_+} L_j' \times [R,\infty) \times \{j\}$$
	for some closed Lagrangian submanifolds $L_1', \dots, L_{k_+}' \subset M$,
\item[(iii)] $V\vert_{[-R,R]\times \mathbb{R}} \subset \mathbb{R}^2 \times M$ is compact.
\end{enumerate}
$V$ is called a Lagrangian cobordism from the Lagrangian family $(L_j')_{j=1, \dots, k_+}$
to the Lagrangian family $\left(L_i\right)_{i=1, \dots, k_-}$, denoted by
$$(L_j')_{j=1, \dots, k_+} \rightsquigarrow \left(L_i\right)_{i=1, \dots, k_-}.$$

\subsection{Lagrangian cobordisms induce cone decompositions.}\label{subsec:LCconedecomposition}
We recall here how a cobordism gives rise to cone decompositions of its ends in $\mathcal{DF}uk(M)$.
Since we work with cohomology, rather than homology, we write here a cohomological reformulation of Theorem A from \cite{BC2}. 
\begin{thm}[Theorem A in \cite{BC2}]
Let $V$ be an oriented cobordism from $L$ to the family $(L_1[l-1], L_2[l-2] \dots , L_l)$.
Assume that all Lagrangians involved (including $V$) are uniformly monotone.
Then there exists a graded quasi-isomorphism
$$L \cong \mathrm{Cone}\left( \dots \mathrm{Cone}(\mathrm{Cone}(L_1 \to L_{2} ) \to L_3)\to \dots \to L_l\right)$$
in the derived Fukaya category $\mathcal{DF}uk(M).$
\end{thm}
Here, we denote by $L[k], k\in \Z$ the Lagrangian $L$ with the same orientation for even $k$, and with oppostite orientation for odd $k$.
(The theorem also holds in the context of $\Z$-gradings, see also \cite{MakWu}.)
 
A special case occurs when there are only three Lagrangians involved, namely $V$ has one right end, $L$, and two left ends, $L_1[1]$ and $L_2$.
Then we get 
\[
L \cong Cone(L_1 \xrightarrow{\varphi} L_2).
\]
As we explain further in the \hyperref[appendix]{appendix}, the morphism $\varphi$ is determined by a unique element $\alpha_V \in HF^0(L_1,L_2)$. In particular, for any Lagrangian $K$ we get a quasi-isomorphism of chain complexes
\[
 \CF^*(K, L) \cong Cone \left( \CF^*(K,L_1) \xrightarrow{\mu^2(\alpha_V,-)} \CF^*(K,L_2) \right).
\]
Note that $\alpha_V$ is independent of $K$.

The associate long exact sequence in cohomology is
\[
\dots \to \HF^{k-1}(K,L) \to \HF^{k}(K,L_1) \xrightarrow{\mu^2(\alpha_V,-)} \HF^k(K,L_2) \to \HF^k(K,L) \to \dots
\]

\section{Mak-Wu cobordism}\label{sec:MW}
We consider a symplectic manifold $(M,\omega)$ and a parametrized Lagrangian sphere $S\subset M$.
Mak-Wu \cite{MakWu} constructed a Lagrangian cobordism $V_{MW}$
with three ends: $S\times S, \Delta$ and $\Gamma_{\tau_S^{-1}}$.
In this section, we will recall the construction of this cobordism, which closely follows \cite{MakWu}.
\subsection{The graph of the Dehn twist.}\label{subsec:MWGraphDT}
Following the principle that surgeries provide cobordisms with three ends \cite[Section 6]{BC1}, the Mak-Wu cobordism also arises as the trace of a surgery. The first step therefore is to understand $\Gamma_{\tau_S^{-1}}$ as the result of a surgery between 
$S \times S \subset M\times M^-$ and the diagonal $\Delta \subset M\times M^-$
along the clean intersection $\Delta_S := (S\times S)\cap \Delta$.
The surgery construction takes place locally in a Weinstein neighbourhood of $S\times S$. We choose a very specific neighbourhood, so that we can later compare it to $\Gamma_{\tau_S^{-1}}$.
Namely, consider the symplectic embedding
\begin{align*}
\widetilde{\varphi} \colon V \times V &\longrightarrow T_\epsilon^*S^n \oplus T_\epsilon^*S^n \subset T^*(S^n \times S^n) \\ 
(x,y) &\longmapsto (\varphi(x), -\varphi(y))
\end{align*}
that identifies $S\times S$ with the zero-section in $T^*(S^n \times S^n)$. Note that 
$$\widetilde{\varphi}^{-1}(N_{\Delta_S}^*) = \Delta \cap (V\times V),$$
where 
\[
    N_{\Delta_S}^*:= \left \{ \alpha \in T^*(S^n \times S^n) \, \vert \,
    \forall v\in \Delta_S \colon \alpha(v) = 0 \right \}.
\]
We will define a surgery model in $T^*(S^n \times S^n)$ for surgery of the zero-section and $N_{\Delta_S}^*$ along their intersection $\Delta_S$.
Then we will glue the surgery model into $V\times V$ via $\tilde{\varphi}$. 
To define the surgery model, we need some auxiliary functions:
\begin{defi}\label{MWGraphDTadm}
A \textit{$\lambda$-admissible} function $\nu_{\lambda}\colon \mathbb{R}_{\geq 0}
\longrightarrow [0, \lambda]$ is a smooth function satisfying
\begin{align*}
\begin{cases}
	 \nu_{\lambda}(0) = \lambda,\\
	 \nu_{\lambda}^{-1} \text{has vanishing derivatives of all orders at } \lambda,\\
	 0 < \nu_{\lambda}(r) < \lambda \text{ and strictly decreasing}
	 									 &\text{for } 0< r < \epsilon,\\
	 \nu_{\lambda}(r)=0                   &\text{for } r\geq \epsilon.
\end{cases}
\end{align*}
\end{defi}

Let $\pi_2 \colon T^*(S^n \times S^n) \cong T^*S^n \oplus T^*S^n \to T^*S^n$ be the projection
to the second summand. Consider 
$\sigma_\pi \colon T^*(S^n \times S^n) \to \R$ defined by
$\sigma_\pi(\xi) = \norm{\pi_2(\xi)}$.
This has a well-defined Hamiltonian flow on $T^*(S^n \times S^n)\backslash \Delta_S$.
Let $\lambda < \pi$. Consider a $\lambda$-admissible function $\nu = \nu_{\lambda}$,
and define the following \textit{flow handle}:
\begin{align*}
H_{\nu} = \left\{ \psi_{\nu(\sigma_\pi(\xi))}^{\sigma_\pi}(\xi) \in T^*(S^n \times S^n) \, \big\vert \,
\xi \in N_{\Delta_S}^* \backslash \Delta_S, \sigma_\pi(\xi) \leq \epsilon\right\}.
\end{align*}
$H_\nu$ can be glued to a part of $S\times S$ and $N^*_{\Delta_S}$, resulting in a smooth Lagrangian in $T^*(S^n \times S^n)$ that coincides with $N_{\Delta_S}^*$ outside of $T_\epsilon^* S \oplus T_\epsilon^*S^n$. 
We denote the resulting Lagrangian by 
\[
(S^n \times S^n) \#_{\Delta_S}^{\nu} N_{\Delta S}^*.
\]
We finally glue this model surgery into $V\times V$:
$$(S\times S) \#_{\Delta_S}^{\nu} \Delta := (\tilde{\varphi})^{-1}
\left((S^n \times S^n) \#_{\Delta_S}^{\nu} N_{\Delta_S}^*\right)
\cup \left(\Delta \backslash (V\times V^-)\right).$$

Mak-Wu \cite[Lemma 3.4]{MakWu} show that all such surgeries are Hamiltonian isotopic for different choices of $\nu$. Moreover, the same construction works for $\nu= \nu_\epsilon^{\text{Dehn}}$ (even though this is \textit{not} admissible)
and the result is again Hamiltonian isotopic to any of the other surgeries.
It's straight-forward to see that
$$(S\times S) \#_{\Delta_S}^{\nu_\epsilon^{\text{Dehn}}} \Delta = \Gamma_{\tau_S^{-1}}$$
and so any of the above surgeries is Hamiltonian isotopic to $\Gamma_{\tau_S^{-1}}$.
In particular, since $\Gamma_{\tau_S^{-1}}$ is relatively symplectically aspherical, so is the surgery $(S\times S)\#_{\Delta_S}^{\nu} \Delta$.

\begin{rem}
This version of surgery is a special case of $E_2$-flow surgery, introduced in \cite[section 2.3]{MakWu} in more general situations.
\end{rem}
\subsection{The cobordism.}\label{subsec:MWcob}
$(S\times S) \#_{\Delta_S}^{\nu} \Delta$ is related to $S\times S$ and $\Delta$ via
a cobordism. 
This follows from a construction called "trace of a surgery", which is a surgery construction in one dimension higher. This was first introduced in \cite{BC1} for the case of a transverse surgery in a point.
As shown in \cite{MakWu}, exactly the same construction works for the $E_2$-surgery along clean intersections. We recall the construction in our special case.

Consider the symplectomorphism
\begin{align*}
\tilde{\varphi} \times \mathrm{id} \colon V \times V\times T^*\R &\longrightarrow T_\epsilon^*S^n \oplus T_\epsilon^*S^n \oplus T^*\R \subset T^*(S^n \times S^n \times \R) 
\end{align*}
and define the handle in the model $T^*(S^n \times S^n \times \R)$ 
as follows:
\[
\hat{H}_\nu = \left \{ \psi_{\nu(\sigma_{\hat{\pi}}(\xi))}^{\sigma_{\hat{\pi}}}(\xi) \in T^*(S^n \times S^n\times \R) \, \big \vert \,
\xi \in N_{\Delta_S \times \{0\}}^* \backslash (\Delta_S \times \{0\}), \sigma_{\hat{\pi}}(\xi) \leq \epsilon \right \},
\]
where $\sigma_{\hat{\pi}} \colon T^*(S^n \times S^n\times \R) \to \R$ is given by
$\sigma_{\hat{\pi}}(\xi_1, \xi_2, p) = \norm{((\xi_2,p)}$.
Here, $\nu= \nu_\lambda$ is a $\lambda$-admissible function, as defined in 
Definition \ref{MWGraphDTadm}.
One computes
\begin{align*}
\psi_t^{\sigma{\hat{\pi}}}(\xi_1, \xi_2,p) = 
\left( \xi_1, \psi^{\sigma}_{\frac{t \norm{ \xi }}{\sqrt{\norm{ \xi } ^2 + p^2}}} (\xi_2),
\psi^{\sigma^{\mathbb{R}}}_{\frac{t \lvert p \rvert}{\sqrt{\norm{ \xi } ^2 + p^2}}} (p) 
\right)
\end{align*}
So more concretely,
$\hat{H}_\nu$ can be described as follows:
\begin{align*}
\hat{H}_{\nu} = \left \{ \left(\xi, \psi^{\sigma}_{\nu\left(\sqrt{\norm{ \xi } ^2 + p^2}\right) \frac{\norm{ \xi}}{\sqrt{\norm{ \xi } ^2 + p^2}}} (\xi),
 \psi^{\sigma^{\mathbb{R}}}_{\nu\left(\sqrt{\norm{ \xi }^2 + p^2}\right) \frac{\norm{ p }}{\sqrt{\norm{\xi } ^2 + p^2}}} (p)\right) 
 \, \Big \vert \,
 \begin{array}{l}
      \xi \in T_{\epsilon}^*S, p \in \mathbb{R},\\
      \sqrt{\norm{ \xi } ^2 + p^2}< \epsilon
 \end{array}
  \right\}.
\end{align*}
Here, $\sigma\colon T^*S \longrightarrow \mathbb{R}$ is the Hamiltonian function
$\sigma(\xi) = \norm{ \xi }$ we used earlier to define $\tau_S$
and $\sigma^{\mathbb{R}}\colon T^*\mathbb{R} \longrightarrow \mathbb{R}$ is given by
$\sigma^{\mathbb{R}}(p) = \vert p \vert$.

The model handle $\hat{H}_\nu$ glues to a part of $(S^n\times S^n \times \R)\backslash \partial H$, which yields the model surgery trace
$$(S^n \times S^n \times \R) \#_{\Delta_S \times \{0\}} N_{\Delta_S \times \{0\}}^*.$$
Gluing this into $M\times M^- \times T^*\R$ via $\tilde{\varphi} \times \mathrm{id}$ we get
\[
V:=
(S\times S \times \R) \#_{\Delta_S \times \{0\}} 
\left( \Delta \times i\R \right):=
(\tilde{\varphi}\times \mathrm{id})^{-1}
\left((S^n \times S^n \times \R) \#_{\Delta_S \times \{0\}} N_{\Delta S \times \{0\}}^*\right). 
\]
$V \subset M\times M^- \times T^*\R$ is a Lagrangian submanifold. Under the identification $T^*\R \cong \C$ via $(q,p) \leftrightarrow q-ip$,
$V$ satisfies
\begin{align*}
&V\cap \pi_{\C}^{-1}(\epsilon) = S \times S \times \{\epsilon\},\\
&V\cap \pi_{\C}^{-1}(i\epsilon) = \Delta \times \{i \epsilon\},\\
&V\cap \pi_{\C}^{-1}(0) = (S \times S) \#_{\Delta_S}^{\nu} \Delta.
\end{align*}
By taking half of $V$, extending it by a ray of $(S \times S) \#_{\Delta_S}^{\nu} \Delta$ at $0$ and smoothing it, and bending the ends, as explained in \cite[Section 6.1]{BC1}, we get a 
cobordism 
$$\tilde{V}\colon (S \times S) \#_{\Delta_S}^{\nu} \Delta \rightsquigarrow 
(S\times S, \Delta) .$$

As discussed in section \ref{subsec:MWGraphDT}, $(S \times S) \#_{\Delta_S}^{\nu} \Delta$
is Hamiltonian isotopic to $\Gamma_{\tau_S^{-1}}$. Gluing a corresponing suspension to $\tilde{V}$ finally gives us the claimed cobordism 
\[
V_{MW}\colon \Gamma_{\tau_S^{-1}} \rightsquigarrow 
(S\times S, \Delta) .
\]

\subsection{Floer theory.}\label{subsec:MWsequence}
Mak-Wu \cite{MakWu} explain how to put gradings on $S\times S$, $\Delta$,
$\Gamma_{\tau_S^{-1}}$ and on $V_{MW}$ such that $V_{MW}$ becomes a graded cobordism from $\Gamma_{\tau_S^{-1}}$ to $(S\times S[1], \Delta)$.
Here, we only use $\Z/2$-gradings, but it follows from their proof, that $V_{MW}$ is an oriented cobordism.

In the situation of symplectically aspherical manifolds, $V_{MW}$
is a relatively symplectically aspherical Lagrangian in $M\times M^- \times \C$ with relatively symplectically aspherical ends.
More precisely, assuming $\omega \vert _ {\pi_2(M)} \equiv 0$
and $\omega \vert _ {\pi_2(M,S)} \equiv 0$ implies that 
$(\omega \oplus -\omega)\vert _{\pi_2(M\times M^-, S\times S)} \equiv 0$,
$(\omega \oplus -\omega)\vert _{\pi_2(M\times M^-, \Delta)} \equiv 0$,
$(\omega \oplus -\omega)\vert _{\pi_2(M\times M^-, \Gamma_{\tau_S^{-1}})} \equiv 0$ and
 $(\omega \oplus -\omega \oplus \omega_{\C}) \vert _ {\pi_2(M\times M^- \times \C, V_{MV})} \equiv 0$.
The latter follows from an argument very similar to the proof of the corresponding result on exactness and monotonicity in \cite[Lemma 6.2, 6.3]{MakWu}.

Floer theory for $V_{MW}$ and the ends is therefore well-defined.
The cone-decomposition result
\ref{subsec:LCconedecomposition}
from Biran-Cornea \cite{BC1}, \cite{BC2} therefore yields a long exact sequence of graded
Lagrangian Floer cohomology groups \cite[Theorem 6.4]{MakWu}:
\begin{align*}
\dots \to \HF^k(K,&S\times S) \xrightarrow{\mu^2(B,-)} \HF^{k}(K,\Delta) \\ &\xrightarrow{\mu^2(A,-)} \HF^{k}(K,\Gamma_{\tau^{-1}}) 
\xrightarrow{\mu^2(C,-)} \HF^{k+1}(K,S\times S) \to \dots
\end{align*}
for any admissible Lagrangian submanifold $K\subset \left(M\times M, \omega \oplus -\omega\right)$.
This is precisely the sequence (\ref{I:eqLES2}).
As indicated, the maps are given by $\mu^2$ operations with elements
$A \in \HF^0(\Delta, \Gamma_{\tau^{-1}})$,
$B \in \HF^0(S\times S, \Delta)$
and $C \in \HF^1(\Gamma_{\tau_S^{-1}},S\times S)$.
$A, B$ and $C$ are independent of $K$.

\begin{prop}
    If $2c_1(M) = 0$ in $\Homol^2(M; \Z)$ and $2c_1(M,S) = 0$
    in $\Homol^2(M,S)$ then $A\neq 0$.
    \footnote{The condition $2c_1(M,S) = 0$ is automatic for $n\geq 2$.
    For $n=1$ it's equivalent to $S$ being a non-contractible circle.}
\end{prop}
\begin{proof}
Under the condition on the Chern class, everything becomes $\Z$-graded,
see \cite{seidel00}.
For $K=\Delta$, the sequence becomes
\begin{align*}
\dots \to \HF^{k}(S,S) \to \HF^k(\mathrm{id}) \xrightarrow{\Psi} \HF^k(\Delta, \Gamma_{\tau^{-1}})
\to \HF^{k+1}(S,S) \to \dots.
\end{align*}
Assume by contradiction that $A=0$. Then $\Psi=0$ and hence
we get $\Z$-graded isomorphisms
\begin{align*}
\Homol^*(S;\Lambda) \cong QH^*(S) \cong \HF^*(S,S) \cong \HF^*(\mathrm{id}) \cong \mathrm{QH}^*(M) \cong \Homol^*(M; \Lambda).
\end{align*}
This is a impossible.
We conclude that $A\neq 0$.
\end{proof}

\section{Symmetry of the Mak-Wu cobordism}\label{sec:INV}
We assume that $\iota^*c\simeq \mathrm{id}$ or $\iota^*c\simeq r$.
In this section, we prove Theorem \ref{Ioutline_of_proof:thm:invariance},
i.e. that $\Phi(V_{MW})$ is Hamiltonian isotopic to $V_{MW}$, where
$
\Phi(x,y,z)= (c(y),c(x),z).
$

\subsection{Linear approximation of anti-symplectic involution.}\label{subsec:INVlinear}
Any map $$c_0 \colon S^n \to S^n$$
induces an anti-symplectic involution
$$c_0^* \colon T^*S^n \to T^*S^n$$
via
$c_0^*(q,p) = (c_0(q), -p\circ (Dc_0)_{c_0(q)})$.
We assume that $c_0 = \mathrm{id}$ or $c_0=r$ (or more generally that $c_0$
is any isometry).
Note that we secretly identify $T^*S^n$ and $T_*S^n$ via the canonical isomorphism $\alpha \colon T_*S^n \to T^*S^n$ coming from the standard Riemannian metric.
The following diagram commutes:
\begin{equation*}
\xymatrix{
	&T^*S \ar[r]^{c_0^*}  &T^*S \\
	&TS \ar[r]^{-Dc_0} \ar[u]_{\alpha}^{\cong} & TS \ar[u]_{\alpha}^{\cong}
}
\end{equation*}
The following Lemma collects some properties of $c_0^*$.
\begin{lem}\label{INVlinear:lem:properties}
Assume that $c_0 = \mathrm{id}$ or $c_0=r$ (or more generally that $c_0$
is any isometry). Then
$c_0^* \colon T^*S \to T^*S$ satisfies
\begin{itemize}
\item $c_0^*(\xi) = - \phi_s^\sigma (-c_0^*(-\phi_s^\sigma (- \xi))) $
\item $\norm{ c_0^* ( - \phi_s^\sigma (-\xi)) } = \norm{ \xi }.$
\end{itemize}
\end{lem}

\begin{proof}
Let $\xi\in T_x^*S\cong T_xS$. Let $\gamma$ be the unique geodesic in $S$ with
$\gamma(0)=x$ and $\gamma'(0)=-\xi$.
Then $\phi_s^{\sigma}(-\xi) = \gamma'(s)$.
Note that
$c_0^*(\gamma'(s)) = -dc(\gamma'(s)) =- (c\circ \gamma)'(s)$.
Moreover, 
$(c\circ \gamma)'(0) = dc(-\xi)= -c_0^*(-\xi)$, 
hence 
\begin{align*}
\phi_s^\sigma ( c_0^* \phi_s^{\sigma}(-\xi))
&= \phi_s^{\sigma}(-\phi_s^{\sigma}(- c_0^*(-\xi)))\\
&= c_0^*(-\xi).
\end{align*}
$c_0^*$ commutes with the minus sign because it is linear.
So the first claim follows.
For the second, note that both $\phi_s^{\sigma}$ and $c_0^*$ both preserve
the length induced by $g$. The latter follows from $c_0$ being an isometry.
\end{proof}

The next proposition shows that any anti-symplectic involution $c$ on $T^*S^n$
as before locally looks like $c_0^*$ for $c_0=\mathrm{id}$
or $c_0= r$.
\begin{prop}\label{INVlinear:prop:main}
    Let $c\colon T^*S^n \rightarrow T^*S^n$ be an anti-symplectic involution restricting to $\sigma\colon S^n \longrightarrow S^n$ where either $\sigma \simeq \mathrm{id}$ or $\sigma \simeq r$. Then for every $\eta_2>0$ there exists 
    $\eta_2 > \eta_1 >0$ and
    a Hamiltonian isotopy
    \[
        \psi_t^H \colon T^*S^n \longrightarrow T^*S^n
    \]
    with $\operatorname{supp}H \subset T^*_{\eta_2}S^n$ such that 
    $$c \psi^H_1 = c_0^* \; \text{ on } \; T^*_{\eta_1}S^n,$$
    where $c_0= \mathrm{id}$ or $c_0=r$.
\end{prop}
\begin{proof}
    Consider the symplectomorphism $\psi:= c \circ \sigma^*$. Write in local coordinates $\psi(q,p)=(u(q,p), v(q,p))$ with $u(q,p) \in S^n$ and
$v(q,p)\in T_{u(q,p)}S^n$. Since $c=T^*\sigma$ on $S^n$ we have
$u(q,0) = q$.
    Consider the following isotopy of symplectomorphisms $\psi_t \colon T^*S^n \to T^*S^n$ between
$\psi_0 = \mathrm{id}$ and $\psi_1(q,p)= \psi$:
\begin{align*}
\psi_{t}(q,p) =
\begin{cases}
 	(u(q,tp), \frac{v(q,tp)}{t}) \qquad \qquad \; t\neq 0\\
 	(u(q,0), (\partial_p v(q,0))p) \qquad t=0
\end{cases}
\end{align*}
$\psi_0(q,p) = (q,p)$
because
as it can be seen from writing $D\psi_{(q,0)}$ in local coordinates:
\begin{align*}
D\psi_{(q,0)} =
    \begin{pmatrix}
        &\mathrm{id} & \partial_q v(q,0)\\
        &\partial_p u(q,0) &\partial_p v(q,0)
    \end{pmatrix}
\end{align*}
and using that $D\psi_{(q,0)}$ is a symplectic matrix, we get
\[
\partial_p v(q,0) = \mathrm{id}.
\]
$\psi_t$ is a Hamiltonian isotopy:
For $n\geq 2$ this is automatic. For $n=1$ it follows from $\psi_t(S^n) = S^n$.
Concatenate $\psi_t$ with a Hamiltonian isotopy $\sigma^* \simeq c_0^*$.
Finally cut off the Hamiltonian so that the resulting Hamiltonian $H$ has support in $T^*_{\eta_2}S^n$. 
Clearly, $\psi_1^H = c \circ c_0^*$ on $T^*_{\eta_1}S^n$ for $\eta_1$
 small enough.
 In particular, $c\psi_1^H=c_0^*$ on $T^*_{\eta_1}S^n$.
\end{proof}
\subsection{The symmetry of the surgery part.}\label{subsec:INVinvariance_Mak_Wu}
Let $V\subset M$ be a Weinstein neighbourhood of $S$ and
$\varphi\colon V \to T_{\delta}^*S^n$ a symplectomorphism
for some $\delta>0$.
Let $0 < \epsilon < \delta$ small enough, such that $c(U) \subset U$ 
for $U:= \varphi ^{-1}(T_{\epsilon}^*S^n)$.
Consider 
$\Phi$
as a map
\[
 U \times U \times \mathbb{C} \to V \times V \times \mathbb{C}.
\]
This induces via $\widetilde{\varphi}\times \mathrm{id}$ the map
\begin{align*}
\Phi \colon T_\epsilon^*S^n \times T_\epsilon^*S^n  \times \mathbb{C} &\to T_\delta^*S^n  \times T_\delta^*S^n \times \mathbb{C} \\
(\xi_1,\xi_2,z) &\mapsto (c(-\xi_2),c(\xi_1),z).
\end{align*}
Let $\psi_t \colon T^*S^n \to T^*S^n$ be the Hamiltonian isotopy from Proposition \ref{INVlinear:prop:main}. Consider the Hamiltonian isotopy
\begin{align*}
\Psi_t\colon T_\delta^*S^n \times T_\delta^*S^n \times T^*\mathbb{R} &\to T_\delta^*S^n \times T_\delta^*S^n \times T^*\mathbb{R} \\
(\xi_1,\xi_2, p) &\mapsto (\psi_t(\xi_1), -\psi_t(-\xi_2),p).
\end{align*}
We consider surgery in $T_{\epsilon}^*S^n$, so that the handle
$\hat{H}_{\nu}$ is contained in $T_{\epsilon}^*S^n$.
We claim that
\begin{itemize}
\item $\Psi_1(\hat{H}_{\nu}) = \Phi(\hat{H}_{\nu})$,
\item $\Psi_t\vert_{S \times S \times \mathbb{R}} = \mathrm{id}$,
\item $\Psi_t(N_{\Delta_S}^* \times\{p\}) = N_{\Delta_S}^* \times \{p\}$ for any $p\in i \R$,
\item $\pi_\C \circ \Psi_t = \pi_\C$.
\end{itemize}
We check these properties:
\begin{itemize}
\item 
Let $\xi \in T^*S$ and $(q,p) \in T^*\R$ such that $$\sqrt{\norm{ \xi } ^2 + p^2}< \epsilon.$$
We introduce the following abbreviations:
\[
    s(\norm{ \xi }, \lvert p \rvert ) = \nu\left(\sqrt{\norm{ \xi } ^2 + p^2}\right) \frac{\norm{ \xi }}{\sqrt{\norm{ \xi } ^2 + p^2}}
\]
and
\[
    r(\norm{ \xi }, \lvert p \rvert ) = \nu\left(\sqrt{\norm{ \xi } ^2 + p^2}\right) \frac{\lvert p \rvert}{\sqrt{\norm{ \xi } ^2 + p^2}}.
\]
Elements of $\hat{H}_{\nu}$ are of the form 
$$\alpha :=(\xi, \psi_{s(\norm{\xi}, \lvert p \rvert)}^\sigma (-\xi), (r(\norm{ \xi }, \lvert p \rvert) + q,p)).$$
Therefore, elements of $\Psi_1(\hat{H}_{\nu})$ are of the form
$$\Psi_1(\alpha) = (cc_0^* (\xi), -cc_0^*(-\psi_{s(\norm{ \xi})}^\sigma (- \xi)), (r(\norm{ \xi }, \lvert p \rvert) + q,p)).$$
Put
$$\zeta:= c_0^*(-\phi_s^\sigma(-\xi)).$$
Then $\norm{ \zeta } = \norm{ \xi }$
by part $2$ of Lemma \ref{INVlinear:lem:properties}. By part $1$ of Lemma \ref{INVlinear:lem:properties} we have the equality
\[
    c_0^*(\xi) = -\psi_s^\sigma(-c_0^*(-\psi_s^\sigma(-\xi)))= - \psi_s^{\sigma}(-\zeta).
\]
Thus
\[
    \Psi_1(\alpha) = (c(-\psi_{s(\norm{ \zeta })}^\sigma (-\zeta)), -c(\zeta), (r(\norm{ \zeta }, \lvert p \rvert ) + q, p))
\]
which are precisely the elements of $\Phi(\hat{H}_{\nu})$. 
\item ${\Psi_t}(S\times S \times \mathbb{R}) = S\times S \times \mathbb{R}$: For $(x,y)\in S \times S$
\begin{align*}
\Psi_t(x,y,p) &= (\psi_t(x), -\psi_t(-y),p) \in S \times S \times \R
\end{align*}
because $\psi_t(S) = S$.
\item $\Psi_t(\xi, -\xi,p) = (\psi_t(\xi),-\psi_t(\xi),p)$ and thus 
 $\Psi_t(N_{\Delta_S}^* \times\{p\}) = N_{\Delta_S}^* \times \{p\}$.
\item $\pi_\C \circ \Psi_t = \pi_\C$ is clear.
\end{itemize}

Thus the surgery model $(S\times S \times \R) \#_{\Delta_S}N^*_{\Delta_S \times \{0\}}$
is Hamiltonian isotopic to the image of itself under $\Phi$.
The smoothing and the extension to a cobordism with ends
$S\times S$, $\Delta$ and $(S\times S) \#_{\Delta_S} \Delta$ can be done while keeping the Hamiltonian isotopy type.
Hence the surgery part of the cobordism is Hamiltonian isotopic to the image of itself under $\Phi$.

\subsection{The symmetry of the suspension part.}\label{subsec:INVinvariance_suspension}
This is very similar to the preceeding surgery part.
Again, it is enough to show the statement for the surgery model.
Let $(S\times S) \#_{\Delta_S}^{\nu_t} \Delta$, $t\in [0,1]$ be a Hamiltonian isotopy, where
all $\nu_t$ are admissible, except for $\nu_1$ which coincides with $\nu_\epsilon^{\mathrm{Dehn}}$.
The Hamiltonian $K_t \colon T^*(S\times S) \to T^*(S\times S)$ generating the isotopy can be chosen to be of the form
$K_t(\xi_1, \xi_2) = K_t(\norm{ \xi_1 },\norm{ \xi_2 })$,
see \cite[Lemma 3.6]{MakWu}. Moreover, $K_t$ can be chosen to be zero near $0$ and near $1$.
The suspension cobordism is the cylindrical extension of the Lagrangian
\[
    \mathcal{S}:= \left \{ (\psi_t^K(x), t-iK_t(\psi_t^K(x))) \in M \times M^- \times \C \, \big \vert \, x\in H^{\nu_0}, t\in [0,1] \right \}.
\]
Consider the Hamiltonian isotopy $\Psi_t$ from before.
We claim that
\begin{itemize}
    \item $\Psi_1(\mathcal{S}) = \Phi(\mathcal{S})$
    \item $\Psi_t \left( {\left((S\times S) \#_{\Delta_S}^{\nu_0} \Delta \right) \times \{p\} } \right) = {\left((S\times S) \#_{\Delta_S}^{\nu_0} \Delta \right) \times \{p\} }$ for $p\in \R_{<0}$
    \item$\Psi_t \left( {\left((S\times S) \#_{\Delta_S}^{\nu_0} \Delta \right) \times \{p\} } \right) = {\left((S\times S) \#_{\Delta_S}^{\nu_1} \Delta \right) \times \{p\} }$ for $p\in \R_{>1}$
    \item $\pi_{\C} \circ \Psi_t = \pi_{\C}$
\end{itemize}
Let us check these properties.
\begin{itemize}
    \item Elements of ``the handle part" of $\mathcal{S}$ can be written as
    $$\alpha =(\xi, \psi^{\sigma}_{\nu(\vert \vert \xi \vert \vert)}(-\xi), t-iK_t(\norm{ \xi }))$$
    for some $\xi\in T_{\epsilon}^*S$.
    Hence elements of the corresponding part of $\Psi_1(\mathcal{S})$ are of the form
    $$\Psi_1(\alpha) = (cc_0^*(\xi), -cc_0^*(-\psi^{\sigma}_{\nu(\norm{ \xi })}(-\xi)), t-iK_t(\norm{ \xi }). $$
    Elements of the corresponding part of 
    $\Phi(\mathcal{S})$ are of the form 
    \[
    (c(-\psi^{\sigma}_{\nu(\norm{ \xi })}(-\zeta)), -c(\zeta), t-iK(\norm{ \eta })
    \]
    for $\zeta \in T_{\epsilon}^*S$. As before,
    using Lemma \ref{INVlinear:lem:properties},
    the elements are in $1:1$-correspondence via 
    $\zeta = c_0^*(-\phi_{\nu(\norm{ \xi })}(-\xi))$.
    
    \item It is very similar, but simpler to see that $\Psi_t$ preserves $H^{\nu_0} \times \{t\}$ for $t\in \R_{<0}$, and also $H^{\nu_1} \times \{t\}$ for $t\in \R_{>1}$.
    
    \item The last item is obvious.
    
\end{itemize}
Therefore, the suspension part $\mathcal{S}$ of the cobordism is Hamiltonian isotopic to 
$\Phi(\mathcal{S})$.

The symmetry of the surgery part shown in section \ref{subsec:INVinvariance_Mak_Wu} and the symmetry of the suspension part shown above together prove Theorem \ref{Ioutline_of_proof:thm:invariance}.

\section{Background on Floer cohomology}\label{sec:ADD}
\subsection{Floer cohomology for symplectomorphisms.}\label{subsec:AdddefHF}
For convenience of the reader we briefly collect the basic ideas and notation for Floer cohomology of a symplectomorphism following \cite{DostSal}. For more detailed expositions, we refer the reader to 
\cite{DostSal} for the monotone case, and to \cite{seidel97}  and \cite{lee} for $W^+$-symplectic manifolds. 

Let $(M,\omega)$ be a closed symplectically aspherical symplectic manifold.
Let $f\in \Symp(M)$ be a symplectomorphism. We first need to choose a Hamiltonian perturbation, namely a family of Hamiltonian functions $\{H_s\colon X \longrightarrow \mathbb{R}\}_{s \in \mathbb{R}}$.
It should be $f$-periodic, in the sense that
$$H_s=H_{s+1} \circ f.$$
Roughly speaking, Floer cohomology of $f$ is Morse cohomology on the twisted loop space
\[
    \Omega_{f} := \left \{ x \in C^{\infty}(\R, X) \; \mid \; x(s+1) = f(x(s)) \right\}
\]
with the closed $1$-form
\[
    \lambda_H(x)(\xi) = \int_0^1 \omega\left(\dot{x}(s)-X^H_s(x(s)), \xi(s) \right) \, ds.
\]
Here, $X^H_s$ denotes the Hamiltonian vector field of $H_s$.
We write $P_{f}(H)$ for the set of $x\in \Omega_{f}$ satisfying $\dot{x}(s) = X_s^H(x(s))$.
For a generic choice of $H$, $P_{f}(H)$ is a finite set. 
The vector space underlying the Floer complex is the $\Lambda$-vector space
generated by $P_{f}(H)$:
\begin{align*}
\CF^*(f;H) = \bigoplus_{x \in P_{f}(H)} \Lambda x.
\end{align*}

$\CF^*(f)$ is $\Z/2$-graded as follows. A generator $x\in P_{f}(H)$ corresponds to a fixed point $x(0)$ of ${f}_H:= (\psi_1^H)^{-1}f$. The degree 
$\mathrm{deg}(x) \in \Z/2$ of $x$
is related to the index of $x(0)$ by
\[
    (-1)^{\mathrm{deg}(x)} =
      \mathrm{sign} \left( \det ( \mathrm{id} - (Df_H)_{x(0)} ) \right).
\]

To define the differential, we need to choose a family of almost complex structures $\mathcal{J}= \{J_s\}_{s\in \R}$ on $M$, compatible with $\omega$ and $f$-periodic, meaning
$J_s= {f}^*(J_{s+1})$.
One considers finite-energy solutions
\[
u\colon \R \times \R \longrightarrow X, (s,t) \longmapsto u(s,t)
\]
of Floer's equation
$$\frac{\partial u}{\partial t} + J_s(u) \left( \frac{\partial u}{\partial s} - X^H_s(u) \right) = 0,$$
which are $f$-periodic in $s$, $u(s+1,t) = f(u(s,t))$,
and satisfy the asymptotic conditions
$$\lim_{t\to -\infty} u(s,t) = x(s)
\text{  and  }
\lim_{t \to \infty} u(s,t) = y(s)
$$
for some Hamiltonian chords $x,y \in \Omega_f$.
Consider the moduli space $\mathcal{M}(x,y;\mathcal{J},H)$ of all such solutions $u$.
For regular $(\mathcal{J},H)$, the moduli space is a smooth manifold.
$\mathbb{R}$ acts on the one-dimensional component $\mathcal{M}^1(x,y;\mathcal{J},H)$ by translation, and the quotient set $\widehat{\mathcal{M}}^1(x,y;\mathcal{J},H) = \mathcal{M}^1(x,y;\mathcal{J},H)/ \mathbb{R}$ is discrete.

The Floer differential $\partial \colon \CF^*(f; \mathcal{J},H) \longrightarrow \CF^*(f; \mathcal{J},H)$
is defined by
\begin{align*}
\partial(x) = \sum\limits_{y \in P_{\varphi}(H)} \sum\limits_{u\in \widehat{\mathcal{M}}^1(x,y;\mathcal{J},H)} y.
\end{align*}
The homology of the chain complex $CF^*(f)$ is called the Floer cohomology  of $f$ with Floer data $(\mathcal{J},H)$ and denoted by $HF^*(f;\mathcal{J},H)$.

There are graded continuation maps for different choices of Floer datum:
Suppose $(H, \mathcal{J})$ and $(H', \mathcal{J}')$ are regular Floer data as above. Choose a family
$(H_{s,t}, J_{s,t})$ 
that satisfies the periodicity assumptions
\[
    J_s = f^*(J_{s+1}) \text{ and } J_s' = f^*(J_{s+1}')
\]
and interpolate between $(H_s, J_s)$ and $(H_s', J_s')$, i.e.
\begin{align*}
    H_{s,t} = H_s', J_{s,t} = J_t' \qquad &\text{for $t$ near $-\infty$},\\
    H_{s,t} = H_s, J_{s,t} = J_t  \qquad &\text{for $t$ near $\infty$}.
\end{align*}
We denote by $\mathcal{M}(x,y; J_{s,t}, H_{s,t})$ the moduli space of solutions to the $1$-parametric Floer equation
\[
\frac{\partial u}{\partial t} + J_{s,t}(u) \left( \frac{\partial u}{\partial s} - X^H_{s,t}(u) \right) = 0
\]
that are $f$-periodic in $s$ and tend to $x$ and $y$ as $t\to \pm \infty$.
For generic choice of $(H_{s,t}, J_{s,t})$ the moduli space is a manifold
and its zero-dimensional component $\mathcal{M}^0(x,y; J_{s,t}, H_{s,t})$
is discrete.
The chain-level continuation map is the chain map
\begin{align*}
    C_{H_{s,t},J_{s,t}} \colon \CF^*(f; \mathcal{J}, H) &\longrightarrow \CF^*(f,\mathcal{J}',H)\\
    x &\longmapsto \sum\limits_{y \in P_{\varphi}(H)} \sum\limits_{u\in \mathcal{M}^0(x,y;J_{s,t},H_{s,t})} y.
\end{align*}
The map induced in cohomology is independent of the choice of homotopy $(H_{s,t}, J_{s,t})$. This allows us to identify the cohomology groups $\HF^*(f,\mathcal{J}, H)$ and $\HF^*(f,\mathcal{J}', H')$ and simply write
$\HF^*(f)$ for the cohomology group of $f$.
\subsection{Lagrangian Floer cohomology.}\label{subsec:ADDdeflagHF}
We recall here Lagrangian Floer cohomology
for relatively aspherical Lagrangians.
Given two closed Lagrangians $L_0, L_1 \subset M$, choose $H$ so that
$\psi_1^H(L_0) \cap L_1$ is a transverse intersection at finitely many points. Then the underlying $\Lambda$-vectorspace of $CF(L_0,L_1; H,J)$ is generated by those points.
The differential is defined by counting $J$-holomorphic strips, using a $w$-compatible almost complex structure
$J$ on $M$.
Floer's equation reads:
\begin{equation*}
    \begin{cases}
      \frac{\partial u}{\partial t} + J_s(u) \left(
      \frac{\partial u}{\partial s} - X_s^H(u) \right) = 0\\
      u(0,t) \in L_0, \qquad u(1,t) \in L_1\\
      \lim_{t\to -\infty} u(s,t) = \psi_s^H(z) \text{ for some } z\in L_0 \\
      \lim_{t\to \infty} u(s,t) = \psi_s^H(w) \text{ for some } w\in L_0 \\
    \end{cases}\,.
\end{equation*}

If $L_0$ and $L_1$ are oriented, we define the degree of $x$ as follows:
\[
 (-1)^{\mathrm{deg}(x)} = (-1)^{\frac{n(n+1)}{2}}\nu(x; L_0,L_1), 
\]
where $\nu(x;L_0,L_1)\in \{\pm 1\}$ denotes the intersection index of $L_0$ and $L_1$ at $x$. This number is defined to be $+1$ if $v_1, \dots , v_{2n}$ is a positive basis
for $T_xM$ whenever $v_1,\dots , v_n$ is a positive basis for $T_xL_0$ and 
$v_{n+1}, \dots , v_{2n}$ is a positive basis for $T_xL_1$.
See \cite[Section 2d]{seidel00} for the grading, and \cite{robbin_salamon}
for the intersection index.
\subsection{Proof of Proposition \ref{HFcompare:prop}.}\label{subsec:ADDcompare}
Choose Floer datum $H_s$ and $J_s$ as in Section \ref{subsec:AdddefHF}.
The generators of $\CF(\phi; H_s, J_s)$ are points
$x\in M$ such that $\phi(x)=\phi_1^H(x)$.
For the Lagrangian Floer complex, we choose the following Floer data:
\begin{align*}
K_s(x,y)= -\frac{1}{2} H_{\frac{1-s}{2}}(x) - \frac{1}{2}H_{\frac{s+1}{2}}(y).
\end{align*}
and
\begin{align*}
\tilde{J}_s := \tilde{J}_s^{(1)} \oplus \tilde{J}_s^{(2)} := J_{\frac{1-s}{2}} \oplus (-J_{\frac{s+1}{2}}).
\end{align*}
Generators of $\CF(\Delta, \Gamma_{\phi}; K_s, \tilde{J}_s)$ are of the form
$(x,\phi(x))\in \psi_1^K (\Gamma_{\mathrm{id}})$.
We show that the map
\begin{align*}
\CF(\phi; H_s, J_s) &\longrightarrow \CF(\Delta, \Gamma_{\phi}; K_s, \tilde{J}_s)\\
x &\longmapsto (x, \phi(x))
\end{align*}
is a chain isomorphism.
This follows from checking that generators get mapped to generators, and
solutions to 
\begin{equation*}
    \begin{cases}
      \frac{\partial v}{\partial t} + \tilde{J}_s(v) \left(
      \frac{\partial v}{\partial s} - X_s^K(v) \right) = 0\\
      v(0,t) \in \Delta \text{ and } v(1,t) \in \Gamma_{\phi}\\
      \lim_{t\to -\infty} v(s,t) = \psi_s^K(z) \text{ for some } z\in \Delta \\
      \lim_{t\to \infty} v(s,t) = \psi_s^K(w) \text{ for some } w\in \Delta \\
    \end{cases}
\end{equation*}
are in one to one correspondence to solutions of
\begin{equation*}
    \begin{cases}
      \frac{\partial u}{\partial t} + J_s(u) \left(
      \frac{\partial u}{\partial s} - X_s^H(u) \right) = 0\\
      u(1,t) = \phi(u(0,t)\\
      \lim_{t\to -\infty} u(s,t) = \psi_s^K(x) \\
      \lim_{t\to \infty} u(s,t) = \psi_s^K(y).  \\
    \end{cases}
\end{equation*}
The correspondence is
given by
\begin{align*}
v(s,t) = (v_1(s,t), v_2(s,t)) \longleftrightarrow u(s,t)=
\begin{cases}
v_1(1-2s), -2t) \qquad s\in [0, \frac{1}{2}]\\
v_2(2s-1,-2t) \qquad s\in [\frac{1}{2}, 1].
\end{cases}
\end{align*}

For the grading: Let $(x,x)\in \Delta \cap \Gamma_{\phi}$.
Let $\mathcal{B}^M$ be a basis of $T_xM$ and consider
the bases $\mathcal{B}^{\Delta}$ and $\mathcal{B}^{\Gamma_{\phi}}$ of 
$T_{(x,x)}\Delta$ and $T_{(x,x)}\Gamma_{\phi}$ associated to $\mathcal{B}^M$.
Note that $\mathcal{B}^{\Delta}$ and $\mathcal{B}^{\Gamma_{\phi}}$ are either both positive or both negative. Hence $\nu(x,x)=1$ if and only if the basis $\mathcal{B}=\left(\mathcal{B}^{\Delta}, \mathcal{B}^{\Gamma_{\phi}} \right)$ is a positive of $T_{(x,x)}M\times M^-$
One computes
\[
        \mathcal{B} = 
        \begin{pmatrix}
            \mathrm{Id} & \mathrm{Id} \\
            \mathrm{Id} & D\phi 
        \end{pmatrix}
        \mathcal{B}_0,
        \]
where $\mathcal{B}_0 = \left( (\mathcal{B}^M,0), (0, \mathcal{B}^M) \right)$.
$\mathcal{B}_0$ is positively oriented if and only if $n$ is even. The determinant of the matrix
is $\det (D\phi - \mathrm{Id}) = \det(\mathrm{Id}- D\phi)$. Hence $$\nu(x,x) = (-1)^n\mathrm{sign} \det (\mathrm{Id}- D\phi)$$ and
\begin{align*}
    (-1)^{\mathrm{deg}(x,x)} &= (-1)^n (-1)^{\frac{2n(2n+1)}{2}}\nu(x,x) \\
    &=(-1)^n (-1)^{\frac{2n(2n+1)}{2}}(-1)^n\mathrm{sign} \det (\mathrm{Id}- D\phi) \\
    &= (-1)^n (-1)^{\frac{2n(2n+1)}{2}}(-1)^n (-1)^{\mathrm{deg}(x)}\\
    &= (-1)^{\mathrm{deg}(x)}.
\end{align*}
This shows that the isomorphism above indeed preserves the grading.


\section*{Appendix A. Algebraic background.}\label{appendix}
We briefly explain the algebraic background relevant for the definition of the the main character of this paper: the element $A \in \HF(\tau^{-1})$.
We follow the conventions for $A_{\infty}$-machinery from \cite{seidelbook}.

Suppose $\mathcal{A}$ is a homologically unital $A_\infty$-category. The Yoneda embedding is a functor
\[
\mathcal{Y} \colon \mathcal{A} \rightarrow mod_{\mathcal{A}}
\]
taking an object $L$ to the $\mathcal{A}$-module 
$\mathcal{Y}(L)$
defined by
\[
\mathcal{Y}(L)(K) := Mor_{\mathcal{A}}(K,L).
\]
and 
\[
\mu^d_{\mathcal{Y}(L)}(b,a_{d-1}, \dots, a_1) := \mu^d(b, a_{d-1}, \dots , a_1)
\]
for $a_i \in Mor_{\mathcal{A}}(K_{i-1},K_i)$, $i\in \{1, \dots , d-1\}$ 
and $b \in \mathcal{Y}(L)(K_{d-1}) = Mor_{\mathcal{A}}(K_{d-1},L)$.

By \cite[Section 2g]{seidelbook} the Yoneda embedding induces a unital, full and faithfull embedding
\[
\Homol(\mathcal{Y}) \colon \Homol(\mathcal{A}) \to \Homol(mod_{\mathcal{A}}).
\]
The derived cateogory $\mathcal{DA}$ of $\mathcal{A}$
can be constructed as follows: Take a triangulated completion of the image of $\mathcal{Y}$ in $mod_{\mathcal{A}}$ and take its homology category.

The following is an immediate consequence of the properties of the Yoneda embedding. 

\begin{cor}
 Each $f\in Mor_{D\mathcal{A}}(\mathcal{Y}(L_1), \mathcal{Y}(L_2))$ can be represented by 
 $\mathcal{Y}(\alpha)$
 for some $\alpha \in Mor_{\mathcal{A}}(L_1,L_2)$. 
 Moreover, $[\alpha]\in Mor_{H(\mathcal{A})}(L_1,L_2)$ is uniquely defined. 
 \end{cor}
 \begin{proof}
First, note that
\[
Mor_{D\mathcal{A}}(\mathcal{Y}(L_1), \mathcal{Y}(L_2))
\cong \Homol(Mor_{mod_{\mathcal{A}}}(\mathcal{Y}(L_1), \mathcal{Y}(L_2))).
\]
 For any object $K$, $\mathcal{Y}(\alpha)$ determines the map
 \[
 \mathcal{Y}(L_1)(K) \cong Mor(K,L_1) \xrightarrow{\mu^2(\alpha,-)}  
 Mor(K,L_2) \cong \mathcal{Y}(L_2)
 \]
 The existence and uniqueness of $\alpha$ follow immediately from $\Homol(\mathcal{Y})$ being full and faithful.
\end{proof}

\noindent
These notions are applied in this paper to the $A_{\infty}$-category $\mathcal{F}uk(M)$.

\section*{Acknowledgements}
This work is part of my doctoral studies at ETH under the supervision of Paul Biran. I would like to express my deep gratitude to Paul Biran for his guidance, many patient explanations and for sharing his insights with me. I'm grateful to Jonny Evans for our conversation about examples. I would also like to thank Alessio Pellegrini for reading this work and helping to improve the paper. 
The author was partially supported by the Swiss National Science Foundation (grant number 200021 204107).


\bibliographystyle{aomalpha}
\bibliography{main}

\end{document}